\documentclass[final]{siamltex}

\usepackage{mathrsfs}
\usepackage{color}
\usepackage{graphicx}
\usepackage{amsmath}
\usepackage{amsfonts}
\usepackage{amssymb}
\usepackage{textcomp}
\usepackage{latexsym}
\usepackage{epstopdf}

\usepackage{threeparttable}

\newtheorem{prop}[theorem]{Proposition}

\newtheorem{remark}[theorem]{Remark}

\def\cal{\mathcal}

\def\E{{\mathcal E}}

\def\T{{\mathcal T}}

\def\N{{\mathcal N}}

\newcommand\ddelta\bigtriangledown
\newcommand\ld\lambda
\newcommand\Ld\Lambda

\def \bZ {\Bbb Z}
\begin{document}
\title{Is $2k$-Conjecture valid for finite volume methods?}
\date{}
\author{Waixiang Cao
\thanks{ Beijing Computational Science Research Center, Beijing, 100084, China.
} \and Zhimin Zhang
\thanks{Beijing Computational Science Research Center, Beijing, 100084, China. Department of Mathematics, Wayne State University,
  Detroit, MI 48202, USA. This author was supported in part by the US National Science
Foundation through grant DMS-1115530.}
\and Qingsong Zou
\thanks{ College of Mathematics and Computational Science and Guangdong Province Key Laboratory of Computational Science, Sun Yat-sen
          University, Guangzhou, 510275, P. R. China. This author is supported in part by
          the National Natural Science Foundation of China under the grant 11171359 and  in part by the Fundamental Research Funds for the Central
         Universities of China.}
}
\maketitle
\begin{abstract}
  This paper is concerned with superconvergence properties of a class of finite volume methods of arbitrary order over rectangular meshes. Our main result is to prove {\it 2k-conjecture}:
  at each vertex of the underlying rectangular mesh, the bi-$k$ degree finite volume solution approximates the exact solution with an order
  $ O(h^{2k})$, where $h$ is the mesh size. As
byproducts,  superconvergence properties for finite volume
discretization errors at Lobatto and Gauss points are also obtained.
All theoretical findings are confirmed by numerical experiments.
\end{abstract}

\section{Introduction}

 As a popular numerical method for  partial differential equations (PDEs), the finite volume method (FVM)
 has a wide range of  applications and attracts intensive theoretical studies, see, e.g., \cite{Bank.R;Rose.D1987,Barth.T;Ohlberger2004, Cai.Z1991,Cai.Z_Park.M2003,Chen.L.Siam,ChenWuXu2011,Emonot1992, Ewing.R;Lin.T;Lin.Y2002,EymardGallouetHerbin2000, Li.R2000,Ollivier-Gooch;M.Altena2002,Plexousakis_2004,Suli1991,Xu.J.Zou.Q2009,zou2009a}
  for an incomplete list of publications.
 However, most theoretical studies in the literature have been focused on linear or quadratic schemes. Recently, arbitrary order FV schemes
 have been constructed and analyzed for elliptic problems in \cite{Cao;Zhang;Zou2012} and \cite{zhang;zou2012}. The
 basic idea of in \cite{Cao;Zhang;Zou2012,zhang;zou2012} to design a FV scheme of any order $k$ is to choose standard finite element space as the
 trial space and construct control volumes with Gauss points in the primal partition. These FV schemes
   are shown to be convergent with optimal rates under  both energy and $L^2$ norms.

    In 1973 Douglas-Dupont proved that the $k$th order $C^0$ finite element method (FEM) to the two-point boundary value problem converges with rate $h^{2k}$ at nodal points. Since then, it has been conjectured (based on many numerical evidences) that the same is true for bi-$k$ finite element approximation under rectangular meshes for the Poisson equation. This conjecture was settled (see \cite{Chen.C.M2012}) recently after almost 40 years. Our earlier study reveals that a class of finite volume methods of arbitrary degree have similar (and even better in some special cases) superconvergence property as counterpart finite element methods in the one dimensional setting \cite{Cao;Zhang;Zou2012,Cao;Zhang;ZouSPR2012}. It is natural to ask whether the {\it 2k-conjecture} is valid for finite volume methods? In this work, we will provide a confirmatory answer to this question.
To be more precise, we shall investigate superconvergence properties
of any order FV schemes studied in \cite{zhang;zou2012}. In
particular, we show that the underlying FVM has all superconvergence
properties of the counterpart FEM.

  We begin with a model problem:
\begin{eqnarray}\label{Poisson}
     - \triangle u =f \ \rm{in}\   \ \Omega,\
     {\rm and}\  u=0,  \ \rm{on}\  \  \partial\Omega,
\end{eqnarray}
  where $\Omega=[a,b]\times[c,d]$  and  $f$ is a real-valued function defined on $\Omega$.

Techniques used in \cite{Cao;Zhang;Zou2012,Cao;Zhang;ZouSPR2012} are very difficult to be applied to FV schemes in the
 two dimensional setting.  Inspired by a recent work \cite{Chen.C.M2012} for the finite element method,
 our approach here is to construct a suitable function
 to correct the error between the exact solution $u$ and its interpolation $u_I$.
Due to different nature of the finite volume method, the construction here
is different from that of for the FEM, some novel design has to be make to serve our purpose.
 In particular, we construct our correction function by designing some special operators,
 instead of a complicated iterative procedure used in the FEM case (see Section 3).
 In addition, using a special mapping from the trial space to test space (\cite{zhang;zou2012}), the FV bilinear form can be regarded as a Gauss quadrature of its corresponding FE bilinear form. Then by taking special cares to the residual term of the Gauss quadrature, we show that our correction function also has desired properties.
Once the correction function is constructed, superconvergence properties at some special points can be obtained with standard arguments.
Our main results can be summarized as the following.

 We first establish superconvergence at nodes :  the bi-$k$ degree FV solution $u_h$ superconverges to $u$ with order $2k$ at any nodal point $P$, i.e.,
  \begin{equation}\label{2kconj}
      (u-u_h)(P)=O(h^{2k}), \qquad (\text{comparing with optimal global rate} \; O(h^{k+1}))
  \end{equation}
  which is termed by Zhou and Lin (\cite{Zhou.Lin}) as {\it 2k-conjecture} in the finite element regime,
 see also, e.g., \cite{Bramble.Schatz.math.com, V.Thomee.math.comp}, for the literature along this line.

 Our superconvergence results also include
 \begin{equation}\label{labottopoint}
 (u-u_h)(L) = O(h^{k+2}), \qquad (\text{comparing with} \; \|u-u_h\|_0 = O(h^{k+1}))
 \end{equation}
 where $L$ is an interior Lobatto point; and
 \begin{equation}\label{gausspoint}
 \nabla (u-u_h)(G) = O(h^{k+1}), \qquad (\text{comparing with} \; \|u-u_h\|_1 = O(h^k))
 \end{equation}
 where $G$ is a Gauss point. As the reader may recall, these rates are the same as the counterpart FEM.

   The rest of the paper is organized as follows. In Section 2,  we present our FV scheme for \eqref{Poisson} and discuss the
   relationship between FV and FE bilinear forms.
    Section 3 is the most technical part, where we construct a correction function and study its properties.
   In Section 4, we prove our main results (\ref{2kconj}) -- (\ref{gausspoint}). Finally, we provide some carefully designed
   numerical examples to support our theoretical findings in Section 5.

   Throughout this paper,  we adopt standard notations for Sobolev spaces such as $W^{m,p}(D)$ on sub-domain $D\subset\Omega$ equipped with
    the norm $\|\cdot\|_{m,p,D}$ and semi-norm $|\cdot|_{m,p,D}$. When $D=\Omega$, we omit the index $D$; and if $p=2$, we set
   $W^{m,p}(D)=H^m(D)$,
   $\|\cdot\|_{m,p,D}=\|\cdot\|_{m,D}$, and $|\cdot|_{m,p,D}=|\cdot|_{m,D}$. Notation``$A\lesssim B$" implies that $A$ can be
  bounded by $B$ multiplied by a constant independent of the mesh size $h$.
  ``$A\sim B$" stands for $``A\lesssim B"$ and $``B\lesssim A"$.

 To end this introduction, we would like to emphasize that this work is a theoretical investigation. Our intention here is not to provide a practical
 method or anything like, rather, we settle a conjecture in convergence rate to the best possible case under very limited special situation.

 Comparing with rich literature on superconvergence of the FEM (see, e.g., \cite{Babuska1996,Bramble.Schatz.math.com,Chen.C.M2001,Chen.C.M_Huang.Y.Q2001,Neittaanmaki1987,V.Thomee.math.comp,Schatz.Wahlbin1996,Wahlbin1995,Zhu.QD;LinQ1989}),
 the superconvergence study for the FVM is still in its infancy, especially for high order schemes.

\section{ Finite volume schemes of arbitrary order}
  \setcounter{equation}{0}
  In this section, we first  recall finite volume schemes
   introduced in \cite{zhang;zou2012}, then we  discuss briefly
   the relationship between the FV  and its corresponding FE bilinear forms.

 Let $\cal T_h$ be a rectangular partition of $\Omega$, where $h$
  is the maximum length of all edges.
  For  any $\tau\in\cal T_h$, we denote by $h^x_\tau,  h^y_\tau$
  the lengths of $x$- and $y$- directional edges of $\tau$, respectively.
  We assume that the mesh $\cal T_h$ is {\it{quasi-uniform}} in the sense that there exist
  constants $c_1,c_2>0$ such that
\[
    h\le c_1h^x_\tau ,\ \ h\le c_2h_\tau^y,\ \ \forall \tau\in\cal T_h.
\]
We denote by $\E_h$ and $\N_h$ the set of edges and vertices of
$\T_h$, respectively.

  We construct control volumes  using Gauss points described below.
   Define reference element $\hat{\tau}=[-1,1]\times[-1,1]$, and $\bZ_{r}=\{1,2,\ldots, r\}, \bZ^0_r=\{0,1,\ldots, r\}$
  for all positive integer $r$.
  Let $G_j, j\in\bZ_k$ be Gauss points of degree $k$ ( zeros of the Legendre polynomial $P_k$) in $[-1,1]$. Then
  $g^{\hat{\tau}}_{i,j}=(G_i,G_j), i,j\in\bZ_k$ constitutes $k^2$ Gauss points in $\hat{\tau}$.
  Given $\tau\in\cal T_h$, let $F_{\tau}$ be the affine mapping from $\hat\tau$ to $\tau$.
 Then Gauss points in $\tau$ are :
 \[
     \cal G_{\tau}=\{ g_{i,j}^{\tau}: g_{i,j}^{\tau}=F_{\tau}(g_{i,j}^{\hat{\tau}}),\ i,j\in\bZ_k\}.
 \]
  Similarly, let
 $L_i, i\in\bZ_k^0$ be Lobatto points of degree $k+1$  on the interval $[-1,1]$, i.e.,
 $L_0=-1, L_{k}=1$ and $L_i,i\in\bZ_{k-1}$ are zeros of  $P_k'$. Then
\[
  \cal N_{\tau}=\{l^{\tau}_{i,j}: l^{\tau}_{i,j} =F_{\tau}(L_i,L_j),\ i,j\in\bZ_{k}^0\}
\]
  constitutes $(k+1)^2$  Lobatto points on $\tau$. We denote by
\[
  \N^{g}=\bigcup_{\tau\in\cal T_h}\cal G_{\tau},\ \  \ \ \N^{l}=\bigcup_{\tau\in\cal T_h}\cal N_{\tau}
\]
  the set of Gauss and Lobatto points  on the whole domain, respectively; and
  $\N^{l}_0$ the set of interior Lobatto points  by excluding Lobatto points on the boundary $\partial\Omega$.
  For any $ P\in \cal N^l_0$,
   the control volume surrounding $P$ is the rectangle $K_{P}^*$  formed by four segments connecting
   the  four Gauss points in ${\cal N}^g$  closest to $P$. Then
\[
  \cal T^{*}_h=\bigcup_{P\in\cal N^l}K_{P}^*
\]
   constitutes a dual partition of $\cal T_h$.


 Next, we denote $\mathbb{P}_k$ as the space of polynomials with degree no more than $k$;
and $\psi_{K_P^*}$, the characteristic function of $K_P^*$. Then
the trial and test spaces are defined as
\begin{eqnarray*}
  U_{h}=\{v\in C(\Omega): v|_{\tau}\in \mathbb{P}_k(x)\times \mathbb{P}_k(y), \tau\in\cal T_h, v|_{\partial
  \Omega}=0\}
\end{eqnarray*}
and
\[
   V_h={\rm{Span}}\{\psi_{K_P^*}: P\in\N_0^l\},
\]
respectively. We  see that $U_h$ is the bi-$k$ degree finite element
space,
  and $V_{h}$ is the piecewise constants space with respect to
  the partition $\cal T_h^*$. They both vanish on the boundary of $\Omega$.

   The finite volume method for solving \eqref{Poisson} is to find $u_h\in U_h$ satisfying the following
   local conservative property
\[
   -\int_{\partial\tau^*}\frac{\partial u_h}{\partial \mathbf{n}}ds=\int_{\tau^*} f dxdy,\ \ \forall \tau^*\in\cal T_h^*,
\]
  or equivalently,
\begin{equation}\label{FVM}
   a_{h}(u_{h}, v_{h})=(f, v_{h}),\ \ \forall v_{h}\in V_{h},
\end{equation}
   where the bilinear form is defined for all $w\in H^{1}_0(\Omega), v_{h}\in V_{h}$ by
\begin{eqnarray} \label{bilinear1}
   a_{h}(w,  v_{h})=-\sum_{E\in\mathcal{E}_{\cal T^*_{h}}}[v_{h}]_E\int_{E}\frac{\partial w}{\partial \mathbf{n}}ds.
\end{eqnarray}
   Here $\mathcal{E}_{\cal T_h^*}$ is the set of interior edges of the
  dual partition $\cal T_h^*$, $[v_h]_E=v_h|_{\tau_2}-v_h|_{\tau_1}$ denotes the jump of
  $v_{h}$ across the common edge $E=\tau_1\cap\tau_2$ of two
  rectangles $\tau_1,\tau_2\in\cal T_h^*$, and $\mathbf{n}$ denotes the normal
  vector on $E$ pointing from $\tau_1$ to $\tau_2$.

   The inf-sup condition and continuity of the bilinear form $a_{h}(\cdot,\cdot)$
   have been established in \cite{zhang;zou2012}.
   Moreover,  we have the following convergence and superconvergence properties.
\begin{lemma} {\rm{(cf.\cite{zhang;zou2012})}}
    Let $u\in H^1_0(\Omega)\cap H^{k+2}(\Omega)$ be the solution of \eqref{Poisson}, and $u_{h}$, the solution of \eqref{FVM}.
    Then,
\begin{equation}\label{super_close}
    |u-u_h|_1\lesssim h^{k}|u|_{k+1},\ \ \ \  |u_h-\tilde{u}_I|_1\lesssim h^{k+1}|u|_{k+2},
\end{equation}
  where $\tilde{u}_I\in U_h$ is the function interpolating $u$ at Lobatto points.
\end{lemma}


We next discuss the relationship between $a_h(\cdot,\cdot)$ and the
 FE bilinear form $a_e(\cdot,\cdot)$, which is
defined for all $v,w\in H^1(\Omega)$ by
\[
a_e(v,w)=\int_\Omega \ddelta v\cdot \ddelta w.
\]
We begin with some necessary notations. Let $A_j, j\in\bZ_k$ denote
the weights of the Gauss quadrature $ Q_k(F)=\sum_{j=1}^k A_j F(G_j)
$ for computing the integral $ I(F)=\int_{-1}^1 F(x) dx$. For all
$\tau\in\cal T_h$ and $v_1, v_2\in L^2(\tau)$,  we define
\[
   \langle
   v_1,v_2\rangle_{\tau}=\sum_{i,j=1}^{k}A_{\tau,i}^xA_{\tau,j}^y(v_1v_2)(g_{i,j}^{\tau}),
\]
where
\[
  A^x_{\tau,j}=\frac{1}{2}h^x_\tau A_j,\  A^y_{\tau,j}=\frac {1}{2}h^y_\tau A_j,\  j\in\bZ_k
\]
are Gauss weights associated with $\tau$.  Then we can define a
discrete inner product on $\Omega$ :
\[
  \langle v_1,v_2\rangle=\sum_{\tau\in
  \cal
  T_h}\sum_{i,j=1}^{k}A_{\tau,i}^xA_{\tau,j}^y v_1(g_{i,j}^{\tau})v_2(g_{i,j}^{\tau}).
\]
 Writing $\partial_x=\frac{\partial}{\partial x},\partial_y=\frac{\partial}{\partial y}$ for simplicity,
 we denote, for all $ w\in H_0^1(\Omega)$,
\[
\partial^{-1}_x w(x,y)=\int_a^x w(x',y) dx',\ \   \partial^{-1}_y w(x,y)=\int_c^y w(x,y') dy'.
\]
  A function $v_h\in V_h$ can be represented as
\[
   v_h=\sum_{P\in\N^l_0}(v_h)_{P}\psi_{K_P^*}=\sum_{P\in\N^l}(v_h)_{P}\psi_{K_P^*},
\]
  where $(v_h)_{P}$  is a constant on the control volume ${K_P^*}$ for $P\in {\cal N}^l$. Here we use the fact
   $(v_h)_{P}=0, P\in\partial\Omega$.

  Furthermore, we denote the (double layer) jump of $v_h$ at the Gauss point $g_{i,j}^{\tau}, \forall\tau\in\T_h, i,j\in\bZ_k$
  as
\[
  \lfloor v_h \rfloor_{g^\tau_{i,j}}=(v_h)_{l^\tau_{i,j}}+(v_h)_{l^\tau_{i-1,j-1}}-(v_h)_{l^\tau_{i-1,j}}-(v_h)_{l^\tau_{i,j-1}}.
\]
With above notations, it is  straightforward to deduce from
\eqref{bilinear1} that
\begin{eqnarray} \label{bilinear2}
a_h(w,v_h)=-\sum_{\tau\in\T_h}\sum_{i,j=1}^k
\left(\partial^{-1}_x\partial_y w+\partial^{-1}_y\partial_x
w\right)(g_{i,j}^\tau)\lfloor v_h \rfloor_{g^\tau_{i,j}}.
\end{eqnarray}

In \cite{zhang;zou2012},  a linear mapping $\Pi : U_h\rightarrow
V_h$
\begin{equation}\label{piv}
  \Pi v=v_h= :\sum_{P\in \N^{l}_0} (v_h)_{P}\psi_{K^*_P}\in {V}_h, \ \  v\in U_h,
\end{equation}
 is defined  by letting
\begin{equation}\label{mapcond}
\lfloor v_h \rfloor_{g^\tau_{i,j}}=A^x_{\tau,i}A^y_{\tau,j}
\partial^2_{xy}v(g^\tau_{i,j}),\ \  \forall g^\tau_{i,j}\in {\cal N}^ g.
\end{equation}
Note that although the number of constraints in \eqref{mapcond}
(which equals to the cardinality of $\N^{{g}}$) is different from
the dimensionality of the test space (which equals to the
cardinality of $\N^{{l}}_0$), it has been rigourously shown in
\cite{zhang;zou2012} that $\Pi$ is well-defined.

With this mapping, we have
\begin{eqnarray*}\label{bilinear}
  a_{h}(w, \Pi v)
  &=&-\langle\partial_x^{-1}\partial_y w, \partial^2_{x,y}
  v\rangle-\langle\partial_y^{-1}\partial_x w, \partial^2_{x,y}
  v\rangle.
\end{eqnarray*}
Since by Green's formula,
 \begin{eqnarray*}
   a_e(w,v)=-\int_{\Omega}\left(\partial_x^{-1}\partial_y w+\partial_y^{-1}\partial_x w\right) \partial^2_{x,y}
  vdxdy,
\end{eqnarray*}
 therefore, the finite volume bilinear form $a_{h}(\cdot,\Pi\cdot)$
  can be regarded as the  Gauss quadrature of the Galerkin bilinear form
  $a_{e}(\cdot,\cdot)$. Note that similar point of view appeared in the analysis of linear FV schemes in \cite{Ewing.R;Lin.T;Lin.Y2002}.

%





\section{Correction function}
\setcounter{equation}{0}

 Superconvergence analysis at a special point can usually be reduced to estimating
\[
a_h(u-u_I, \Pi v), \forall v\in U_h,
\]
where $u_I\in U_h$ is an interpolant of $u$  which will be defined
in \eqref{interpolation}. A straightforward analysis using the
continuity of $a_h(\cdot,\cdot)$ results in
\[
|a_h(u-u_I, \Pi v)|\lesssim h^k,
\]
due to the restriction of optimal error bound
\[
|u-u_I|_1 \lesssim h^k.
\]
Further analysis based on standard superconvengence argument may lead to
\[
|a_h(u-u_I, \Pi v)|\lesssim h^{k+1},
\]
an improvement by order one, but is still far from our need. To
obtain desired superconvergence results, more delicate analysis is
necessary. In this section, we shall construct a correction function
$w_h$ with  following  properties.

\begin{prop}\label{theorem_ver}
    Assume  that $u\in H^{\alpha+1}(\Omega), \alpha=k+2 (or \  2k)$.
    Then there exists a function $w_{h}\in U_{h}$ such that $w_{h}=0$ at all
    nodes and
\begin{equation}\label{ineq_1}
   \|w_{h}\|_{\infty}\lesssim h^{k+2}|{\rm{ln}}h|^{\frac 12}\|u\|_{\alpha+1}.
\end{equation}
  Furthermore,
\begin{equation}\label{correction_2}
   |a_{h}(u-u_I-w_{h}, \Pi v)|\lesssim h^{\alpha}\|u\|_{\alpha+1}\|v\|_1,\ \forall
   v\in U_{h}.
\end{equation}
\end{prop}

In the rest of this section, we will first construct $w_h$ and then
verify that $w_h$ satisfies Proposition \ref{theorem_ver}.



\subsection{Construction}
 In this  subsection, we construct a suitable correction function
$w_h$ by introducing some special operators.
 Our device is much transparent and simpler than that in
\cite{Chen.C.M2012} for the finite element method, where a complex
iterative procedure is used.

We begin with notations and preliminaries.
 Since $\T_h$ is a partition of rectangles, there exist $a=x_0< x_1<\ldots,<x_m=b$ and $c=y_0< y_1<\ldots,<y_n=d$
  such that
\[
    \cal T_h=\{\tau_{i,j}:
    \tau_{i,j}=[x_{i-1},x_i]\times[y_{j-1},y_j], i\in\bZ_m,
    j\in\bZ_n\}.
 \]
  We denote by
  $B_i^x=[x_{i-1},x_i]\times [c,d], i\in\bZ_m $, the element-band
  along $x$-direction and $B_j^y= [a,b]\times[y_{j-1},y_j], j\in\bZ_n $,
  the element-band along $y$-direction, respectively.  For any rectangle $B\subset\Omega$, we define
\[
   U_{h}(B)=\{v\in C(\Omega): v|_{B}\in \mathbb{P}_k(x)\times \mathbb{P}_k(y), v|_{\partial
  B}=0\}.
\]
Note that when $k=1$, $U_{h}(B)=\{0\}$.


For all $i\in \bZ_m$, let
  ${\cal L}_{B_i^x} : H^1_0(\Omega)\rightarrow U_h(B_i^x)$ be the
  operator which maps
  $w\in H^1_0(\Omega)$ to ${\cal L}_{B_i^x}(w)$ defined by
\begin{equation}\label{correction_1}
   a_{h}(\cal L_{B_i^x}(w), \Pi v)=-\langle\partial_y^{-1}\partial_x w,\partial^2_{x,y}v\rangle_{B_i^x},\ \ \forall v\in
   U_{h}(B_i^x).
\end{equation}

  Note that on one hand, given $w\in H^1_0(\Omega)$,
\[
  -\langle\partial_y^{-1}\partial_x
   w,\partial^2_{x,y}v\rangle_{B_i^x},\ \ \forall v\in U_h(B_i^x)
\]
  is a bounded linear functional on $U_h(B_i^x)$.
  On the other hand,
  the coercivity and continuity of the bilinear form
  $a_h(\cdot,\Pi\cdot)$ have been established in \cite{zhang;zou2012}. Then by the Lax-Milgram Lemma,
  \eqref{correction_1} has a unique solution and
thus the operator $\cal L_{B_i^x}$ is well defined.

 We define  a global operator ${\cal L}^x:
  H^1_0(\Omega)\rightarrow U_h$ by
\[
  \cal L^x(w)|_{B_i^x} : =\cal L_{B^x_i}(w),\ \ \forall i\in\bZ_m.
\]
  Since $\cal L_{B_i^x}(w)=0$ on the boundary $\partial B_i^x$,
   $\cal L^x(w)=0$ on all $\partial B_i^x, i\in\bZ_m$. Consequently,
  $\cal L^x(w)=0$ at all vertices.

  By a slight modification, we can define another operator
  $\tilde{\cal L}^x
:H^1_0(\Omega)\rightarrow U_h$  by letting
\[
  \tilde{\cal L}^{x}(w)|_{B_i^x} : =\tilde{\cal L}_{B_i^x}(w) ,\ \  \forall
  i\in\bZ_m,
\]
 where the local operator  $\tilde{\cal L}_{B_i^x}: H^1_0(\Omega)\rightarrow U_h(B_i^x)$
 is defined by
\begin{equation}\label{correction4}
   a_h(\tilde{\cal L}_{B_i^x}(w),\Pi v)=-\langle \partial_x^{-1}\partial_yw,\partial^2_{x,y}v\rangle_{B_i^x},\ \ \forall v\in U_h(B_i^x).
\end{equation}

 By the same token, we define
  ${\cal L}_{B_j^y}$, $\tilde{\cal L}_{B_j^y}$, ${\cal L}^y$, and $\tilde{\cal L}^y$.

  Next we define some projectors. Let $P_r, r\ge 0$ be the Legendre polynomial of degree $r$ and
denote by
\[
   \phi_0(t)=\frac{1-t}{2}, \ \ \phi_1(t)=\frac{1+t}{2}, \ \ \phi_{r+1}(t)=\int_{-1}^t P_r(s)ds,\ r\ge 1,
\]
  the series of Lobatto polynomials on the interval $[-1,1]$.
With these Lobatto polynomials, we have the following expansion for
all $v\in H^1(\Omega)$ and $(x,y)\in B_i^x,i\in\bZ_m$ along
$x$-direction
\[
    v(x,y)=\sum_{r=0}^{\infty}b_r(y)\phi_r(s),
\]
 where $s=(2x-x_i-x_{i-1})/h_i^x\in[-1,1]$,
\[
  b_0(y)=v(x_{i-1},y),\  b_1(y)=v(x_i,y),\\
\]
 and
\begin{equation}\label{esti_coeffi}
   b_r(y)=\frac{2r-1}{2}\int_{-1}^{1}\partial_s{v}(x,y)\phi'_r(s)ds,\ \
   r\ge 2.
\end{equation}
   Next, we define a projector $Q^x_p, p\ge 1$ along the $x$-direction. Given $(x,y)\in \Omega$, there exists an $i\in \bZ_m$
  such that $(x,y)\in B_i^x$,  we then define
\[
   (Q^x_pv)(x,y)=\sum_{r=0}^{p}b_r(y)\phi_r(s).
\]
Obviously, $Q^x_p, p\ge 1$ is a bounded operator and $Q^x_pv=v$ for
all $v(\cdot,y)\in \mathbb{P}_{p} $. Consequently, by
 the Bramble-Hilbert lemma, there holds for all $(x,y)\in B_i^x$
\begin{equation}\label{eq9}
  |(v-Q^x_pv)(x, y)|\lesssim
  h^{p}\int_{x_{i-1}}^{x_i}|\partial_x^{p+1}v(x,y)|dx
\end{equation}
   and
\begin{equation}\label{deq}
  |\partial_x(v-Q^x_pv)(x, y)|\lesssim
  h^{p-1}\int_{x_{i-1}}^{x_i}|\partial_x^{p+1}v(x,y)|dx.
\end{equation}
  These inequalities will be  frequently used in our later
  analysis. Moreover, by the properties of Legendre and Lobatto polynomials,
\begin{equation}\label{orth-prop}
  \partial_x(v-Q^x_pv)(\cdot,y) \bot \mathbb{P}_{p-1},\  \ \ (v-Q^x_pv)(\cdot,y)\bot
  \mathbb{P}_{p-2},\forall y\in [c,d],
\end{equation}
where $\mathbb{P}_{-1}=\emptyset$.
 Noticing that $\phi_r(\pm 1)=0, r\ge 2$, we have
\begin{equation}\label{interqx}
  (Q^x_pv)(x_i,y)=v(x_i,y),\ \ (Q^x_pv)(x_{i-1},y)=v(x_{i-1},y),\
  \forall y\in[c,d].
\end{equation}

  The projector $Q^y_p,p\ge 1$ along $y$-direction   can be defined similarly.
   With \eqref{interqx} and counterpart properties in the $y$-direction, we define an interpolation
\begin{equation}\label{interpolation}
  v_I=Q^x_kQ^y_kv
\end{equation}
and the residuals
\[
\ \  E^xv=v-Q^x_kv,\ \ E^yv=v-Q^y_kv,
\]
  then we have
\[
 (v-v_I)(P)=0, \forall P\in\N_h,
\]
   and
\begin{equation}\label{decomposion}
   v-v_I=E^xv+E^yv-E^yE^xv.
\end{equation}


We are now in a perfect position to construct our correction
 function $w_h$. Let
\begin{equation}\label{correction}
   w_h=\cal L^{x}(E^xu)+\cal L^{y}(E^yu)+\tilde{\cal L}^{x}(E^xu)+\tilde{\cal L}^{y}(E^yu)-\cal L^{y}(E^yE^xu)-\tilde{\cal L}^{y}(E^yE^xu).
\end{equation}

 Obviously, $w_h\in U_h$ and $w_h(P)=0$ for all  $P\in\N_h$.


\subsection{Analysis}

   In this subsection, we shall prove $w_h$ defined   by \eqref{correction} satisfies all  properties listed in Proposition
  \ref{theorem_ver}.  For simplicity,  we  assume  in this subsection that
\[
   h=h_{\tau}^x=h_{\tau}^y,\ \ \forall \tau\in\cal T_h.
\]

  Consider  $\cal
  L^x(E^xu)$,  the first term of $w_h$. For this purpose, we  need to present \eqref{correction_1} in
  its linear algebraic form.
 We begin with a presentation of a basis of $U_h(B_i^x), i\in\bZ_m$.
For all $(x,y)\in B_i^x$ and $0\le p,q\le k$, let
\begin{equation}\label{basefun1}
 \Psi_{p,q}(x,y)=
\phi_p(s)\phi_q(t),
\end{equation}
 where
\[
   s=(2x-x_i-x_{i-1})/h,\ \ t=(2y-d-c)/(d-c).
\]
Then the function system $\{\Psi_{p,q},2\le p,q\le k\}$ constitutes
a basis of $U_{h}(B_i^x)$.
 Since $\cal L_{B_i^x}(E^xu)\in U_h(B_i^x)$, we have the representation
\[
\cal L_{B_i^x}(E^xu)=\sum_{p,q=2}^{k}w_{p,q}\Psi_{p,q}.
\]
  Let
\[
    D=(d_{p,q})_{(k-1)\times(k-1)},\ \
    K=(m_{p,q})_{(k-1)\times(k-1)},
\]
  where
\[
   d_{p,q}=\langle \phi'_p,\phi'_q\rangle_{[-1,-1]}, \ \ \ \ m_{p,q}=-\langle\partial^{-1}\phi_{p},\phi'_{q}\rangle_{[-1,-1]},\  2\le p,q\le
   k
 \]
    with the discrete inner product defined by
\[
  \langle v_1,v_2\rangle_{[-1,-1]}=\sum_{r=1}^kA_rv_1(G_r)v_2(G_r).
\]
   By\cite{DavisRabinowitz1984}(p98, (2.7.12)),
\begin{equation}\label{gaudruature_error}
   \langle v_1,v_2\rangle_{[-1,-1]}=\int_{-1}^1 (v_1v_2)(x)
   dx-c_k(v_1v_2)^{(2k)}(\xi),
\end{equation}
  where
  $c_k=\frac{2^{2k+1}(k!)^4}{(2k+1)[(2k)!]^3}$ and $\xi\in(-1,1)$.
  Taking $v=\Psi_{r,l}, r,l=2,\ldots, k$ in \eqref{correction_1},
 we derive
\begin{equation}\label{equa1}
    \sum_{p,q=2}^{k}\Big((d-c)^2d_{p,r}m_{q,l}
    +h^2d_{q,l}m_{p,r}\Big)w_{p,q}=f_{r,l},
\end{equation}
 where
\begin{equation}\label{eq5}
    f_{r,l}=-(d-c)h\langle\partial_y^{-1}\partial_xE^xu, \partial^2_{x,y}\Psi_{r,l}\rangle_{B_i^x}.
\end{equation}

 Denote the unknowns $X=(X_2,\ldots
   X_k)^{T}$ and the right-hand side $F=(F_2,\ldots, F_k)^{T}$ with vectors
\[
   X_r=(w_{r,2},\ldots, w_{r,k})^{T},\ F_r=(f_{r,2},\ldots,
   f_{r,k})^T,\
   r=2,\ldots,k.
\]
  Then \eqref{equa1} can be rewritten as
\begin{equation}\label{matrix_form}
  \Big((d-c)^2(D\otimes K)+h^2(K\otimes D)\Big)X=F,
\end{equation}
  where for two matrices $B_1=(b^1_{p,q})_{k\times k}$ and $B_2=(b^2_{p,q})_{k\times k}$,
 the tensor product $B_1\otimes B_2$ is a matrix of $k^2\times k^2$ defined by
\[
   B_1\otimes B_2= (B_{p,q})_{k\times k},\ \ B_{p,q}=b^1_{p,q}B_2,\ \ \forall\  p,q\le k.
\]

\medskip
 With the linear system \eqref{matrix_form}, the study of the properties of $\cal
  L^x(E^xu)$ is reduced to the estimation of the vector $F$ and the matrix $A=(d-c)^2(D\otimes K)+h^2(K\otimes D)$.

    We first estimate the vector $F$.
\begin{lemma}\label{lemma1}
   If $u\in H^{\alpha+1}(\Omega),\alpha\ge k+1$, then
\begin{equation}\label{eq8}
  \|F_p\|_{\infty}\lesssim h^{\min(\alpha,2k+2-p)}|{\rm{ln}}h|^{\frac 12}\|u\|_{\alpha+1,B_i^x},\ \ p=2,\ldots, k.
\end{equation}
\end{lemma}
\begin{proof}
 For all $\tau\in B_i^x, i\in\bZ_m$,
  we  denote Gauss points  $g_{r,l}^\tau=(g_{\tau,r}^x, g_{\tau,l}^y), r,l\in\bZ_k $.
  Let  $\Theta=\partial_y^{-1}\partial_x (Q^x_{\alpha-1}E^xu)$.
  Note that for any fixed $y$, $\partial^2_{x,y}\Psi_{p,q}(\cdot,y)\in{\mathbb P}_{k-1}$, by
  the orthogonality \eqref{orth-prop} and the fact that $\Theta=\partial_x (Q^x_{\alpha-1}E^x(\partial_y^{-1}u))$,
  we have
\[
  \int_{x_{i-1}}^{x_i}\Theta\partial^2_{x,y}\Psi_{p,q} dx=
  \int_{x_{i-1}}^{x_i}(\partial_xE^x(\partial_y^{-1}u))\partial^2_{x,y}\Psi_{p,q}
  dx=0,
\]
  thus
\[
   \langle\Theta,\partial^2_{x,y}\Psi_{p,q}\rangle_{B_i^x}=-\sum_{\tau\in
      B^x_i}\sum_{l=1}^kA_{\tau,l}^ye^{\tau}_{p,q}(g_{\tau,l}^y),
\]
 where
\[
   e^{\tau}_{p,q}(y)=\int_{x_{i-1}}^{x_i}\Theta\partial^2_{x,y}\Psi_{p,q} dx-
   \sum_{r=1}^{k}A_{\tau,r}^x\left(\Theta\partial^2_{x,y}\Psi_{p,q}\right)(g_{\tau,r}^{x},y)
\]
  is the error of  Gauss quadrature  for calculating the integral of $\Theta\partial^2_{x,y}\Psi_{p,q}$ in
  $[x_{i-1},x_i]$.
  By \eqref{gaudruature_error}, there exists
  a point $\xi_i\in(x_{i-1},x_i)$ such that
\[
e^{\tau}_{p,q}(y)={c_k}\frac{h^{2k+1}}{2^{2k+1}}\partial^{2k}_{x}\left(\Theta\partial^2_{x,y}\Psi_{p,q}\right)(\xi_i,y).
\]
   Note that
\[
  \partial_y\Psi_{p,q}=\phi'_q=O(1),\ \ \
  \partial^{(r)}_x\Psi_{p,q}=\big(\frac{2}{h}\big)^r\phi^{(r)}_p=O( h^{-r}),\ \ \forall\  r\le p
\]
  and
\[
   \|\partial_{x}^{j}\Theta\|_{\infty,B_i^x}\lesssim\|\partial_{x}^{j+1}E^x(\partial_y^{-1}u)\|_{\infty,B_i^x}\lesssim |u|_{\alpha-1,\infty, B_i^x},\ \
   \forall\  j<\alpha-1.
\]
   Then, by  the Leibnitz formula for derivatives,
\[
  |e^{\tau}_{p,q}|\lesssim h^{2k+1-p}|u|_{\alpha-1,\infty,B_i^x},\  2\le q\le
  k,
\]
  which implies
\begin{equation}\label{equation1}
    \left|\langle\Theta,\partial^2_{x,y}\Psi_{p,q}\rangle_{B_i^x}\right|\lesssim h^{2k+1-p}|u|_{\alpha-1,\infty,B_i^x}.
\end{equation}
  On the other hand, by the  approximation property of $Q_p^x, p\ge 1$,
\[
   |E^xu-Q^x_{\alpha-1}E^xu|_{1,\infty,B_i^x}\lesssim
   h^{\alpha-1}|E^xu|_{\alpha,\infty,B_i^x}\lesssim
   h^{\alpha-1}|u|_{\alpha,\infty,B_i^x}.
\]
   Consequently,
\begin{equation}\label{equation1.2}
   \left|\langle\partial_y^{-1}\partial_x(E^xu-Q^x_{\alpha-1}E^xu),\partial^2_{x,y}\Psi_{p,q}\rangle_{B_i^x}\right|\lesssim
   h^{\alpha-1}|u|_{\alpha,\infty,B_i^x}.
\end{equation}
    Furthermore, by the definition \eqref{eq5},
\[
  -\frac{f_{p,q}}{h(d-c)}=\langle\partial_y^{-1}\partial_x(E^xu-Q^x_{\alpha-1}E^xu),\partial^2_{x,y}\Psi_{p,q}\rangle_{B_i^x}+
    \langle\Theta,\partial^2_{x,y}\Psi_{p,q}\rangle_{B_i^x}.
\]
Substituting \eqref{equation1} and \eqref{equation1.2} into the
above equation,  we obtain
\[
   |f_{p,q}|_{\infty}\lesssim
   h^{\min(\alpha,2k+2-p)}\|u\|_{\alpha,\infty,B_i^x}.
\]
   Now recall from standard regularity argument \cite{adams},
\[
   \|u\|_{\alpha,\infty,B_i^x}\lesssim |{\rm{ln}}h|^{\frac 12}\|u\|_{\alpha+1,B_i^x},
\]
   the desired estimate \eqref{eq8} follows.
\end{proof}
\medskip

  We next study  properties of the matrix $A=(d-c)^2(D\otimes K)+h^2(K\otimes D)$.
By the orthogonality of Legendre polynomials and the fact that
$k$-point Gauss quadrature is exact for polynomials of degree
$2k-1$, we have
\[
   d_{p,q}=(P_{p-1},P_{q-1})=0, p\neq q, \ \ \
   d_{p,p}=\frac{2}{2p-1},\ \ p,q=2,\ldots, k.
\]
 In other words, $D$ is a diagonal matrix. Similarly,
\begin{eqnarray}\label{mij}
  m_{p,q}=-(\partial^{-1}\phi_p,\phi'_q)=(\phi_p,\phi_q),\ p,q\le k, p+q\le
  2k-1.
\end{eqnarray}
  By the quasi-orthogonal property of Lobatto polynomials,
 ${m}_{p,q}\neq 0$ only when $p-q=0,\pm 2$.
 Consequently, $K$ is a five-diagonal matrix.

\begin{lemma}
   The matrix $K$ is symmetric and  positive definite.
\end{lemma}
\begin{proof}
Let $K_1=(m^1_{p,q})_{(k-1)\times(k-1)}$ with $
m^1_{p,q}=(\phi_p,\phi_q),p,q=2,\ldots, k.$ By \eqref{mij},
\[
  m_{p,q}^1= m_{p,q}, \ \forall\ p,q\le k, p+q\le
  2k-1.
\]
   We next study the relationship of $m_{k,k}^1$ and $m_{k,k}$.
   Denoting
\[
  e_k=m_{k,k}-m^1_{k,k},
\]
  we have from \eqref{gaudruature_error} and the Leibnitz formula  for
  derivatives
\[
  e_k=c_k((\partial^{-1}\phi_k)\phi'_k)^{(2k)}(\xi)=c_k
  \binom{2k}{k-1}\|\phi_k\|^2_{k,\infty}>0.
\]
  Then
\[
   K=K_1+K_2,
\]
 where $K_2=(m^2_{p,q})_{(k-1)\times(k-1)}, p,q=2,\ldots,k$ with
\[
   m^2_{k,k}=e_k>0,\ \ m^2_{p,q}=0, \ {\rm{otherwise}}.
\]
  Since $K_1$ is  symmetric and positive definite,
 $K$ is also symmetric and positive definite.
\end{proof}

    Note that both $D$ and $K$ are symmetric and positive definite and
independent of $h$, then both $D\otimes K$ and $D\otimes K$ are also positive definite. By the definition of $A$, we have
\[
\det(A)=\det((d-c)^2(D\otimes K)) +O(h^2).
\]
 Therefore, when $h$ is sufficiently small, $\det A$
 is positive and uniformly bounded from below. In other
words, when $h$ is sufficiently small
\begin{equation}\label{detb}
0< \det(A)^{-1}\lesssim C,
\end{equation}
where $C$ is independent of $h$.

\medskip

 With the estimate for $F$ and properties
of $A$, we are now ready to estimate $\cal L_{B_i^x}(E^xu)$.

\begin{lemma} Assume $u\in H^{\alpha+1}(\Omega), \alpha=k+2(or \  2k)$.
Then for sufficiently small $h$ and all $i\in\bZ_m$
\begin{equation}\label{eq6}
    \|X_r\|_{\infty}\lesssim h^{k+2+\max(0,\alpha-k-r)}|{\rm{ln}}h|^{\frac 12}\|u\|_{\alpha+1, B_i^x},\ \ r=2,\ldots,
    k.
\end{equation}
   Consequently,
\begin{equation}\label{eq66}
   \left\|\cal L_{B_i^x}(E^xu)\right\|_{\infty}\lesssim
   h^{k+2}|{\rm{ln}}h|^{\frac 12}\|u\|_{\alpha+1,B_i^x}.
\end{equation}
\end{lemma}
\begin{proof}
   Note that
\[
    \left\|\cal
    L_{B_i^x}(E^xu)\right\|_{\infty}\lesssim\sum_{r=2}^k\|X_r\|_{\infty},
\]
   then \eqref{eq66} follows from \eqref{eq6}.
   We next show \eqref{eq6}.  When $u\in
  H^{k+3}(\Omega)$, By \eqref{eq8}, \eqref{detb} and the Cramer's rule, we have
\[
     \|X_r\|_{\infty}\lesssim h^{k+2}|{\rm{ln}}h|^{\frac 12}\|u\|_{k+3,B_i^x},\ \  r=2\ldots,k.
\]
   Then \eqref{eq6} is valid for $\alpha=k+2$. To
   prove \eqref{eq6} for the case $\alpha=2k$,
  we rewrite $A$ in its block matrix form $
A=(A_{r,l})_{(k-1)\times(k-1)}$,
  where each
\[
  A_{r,l}=(d-c)^2d_{r,l}K+h^2m_{r,l}D,\ \ \ r,l=2,\ldots,k
\]  is a
 $(k-1)\times(k-1)$ matrix. Let
\[
   A'_{r,l}=A_{r,l}h^{-|r-l|}
 \]
   and
\[
   Y_r= X_rh^{r-2k-2}|{\rm{ln}}h|^{-\frac 12}\|u\|^{-1}_{2k+1,B_i^x},\ \ F_r'=F_r h^{r-2k-2}|{\rm{ln}}h|^{-\frac
   12}\|u\|^{-1}_{2k+1,B_i^x}.
\]
Then both $A'_{r,l}$ and $F_r'$ are independent of
   $h$.
    By \eqref{eq8}, we have
\[
   \|F_r\|_{\infty}\lesssim h^{2k+2-r}|{\rm{ln}}h|^{\frac 12}\|u\|_{2k+1,B_i^x}.
\]
Multiplying the $r$-th
  equation of \eqref{matrix_form} with the factor $ h^{r-2k-2}|{\rm{ln}}h|^{-\frac 12}\|u\|^{-1}_{2k+1,B_i^x}$ , we
  have for all $r=2,\ldots,k$
\begin{equation}\label{linear_eq}
    h^4A'_{r,r-2}Y_{r-2}+A'_{r,r}Y_r+A'_{r,r+2}Y_{r+2}=F'_r,
\end{equation}
  where we use the notations $A_{2,0}=A_{3,1}=A_{k-1,k+1}=A_{k,k+2}=0.$
  Let $B=(B_{r,l})_{(k-1)\times(k-1)}$ with
\[
    B_{r,l}=A'_{r,l},\ \ r\le l,\ \ B_{r,l}=h^4A'_{r,l},\ \
    \rm{otherwise}.
\]
  Then \eqref{linear_eq} can be written as a linear system $BY=F'$.
  A direct calculation yields
\[
  det(B)=\prod_{r=2}^kdet(A'_{r,r})+O(h^4),
\]
 which means  that $B$ is uniformly
bounded from below.
  By Cramer's rule, each entry of $Y$ is bounded independent of $h$. In other words,
  $\|Y_r\|_{\infty}\lesssim 1$.  Consequently,
\[
   \|X_r\|_{\infty}\lesssim h^{2k+2-r} |{\rm{ln}}h|^{\frac 12}\|u\|_{2k+1,B_i^x},\ \
   r=2,\ldots,k.
\]
   This finishes our proof.
\end{proof}


%


\medskip

 To prove Proposition \ref{theorem_ver}, we still need to study the
 residual
\begin{eqnarray*}
   (R_i^x(w),v)=-\langle\partial^{-1}_y\partial_x w, \partial^2_{x,y}v\rangle_{B_i^x}-a_{h}(\cal L_{B_i^x}(w), \Pi v),\ \
 w\in H^1_0(\Omega)
\end{eqnarray*}
  for a general function $ v\in U_{h}$. Note that when $v\in
  U_h(B_i^x)$, we have
 $(R_i^x(w),v)=0.$

\begin{lemma}
 Assume  that $u\in H^{\alpha+1}(\Omega), \alpha=k+2(or \  2k)$.  Then for a general function $ v\in U_{h}$,
  \begin{equation}\label{eq7}
   |(R_i^x(E^xu),v)|
   \lesssim h^{\alpha}\|u\|_{\alpha+1, B_i^x}\|v\|_{1,B_i^x},\ \ \forall i\in\bZ_m.
\end{equation}
\end{lemma}
\begin{proof}
  Note that  for all $B_i^x\subset\Omega, i\in\bZ_m$,
\[
   \phi_0(s)=\frac{x_i-x}{h},\ \ \phi_1(s)=\frac{x-x_{i-1}}{h}\notin
   U_h({B_i^x}),
\]
  where $s=(2x-x_i-x_{i-1})/h\in[-1,1]$.
 Then a general function $ v\in U_{h}$ has the decomposition
\[
 v(x,y)={v_h}(x,y)+\tilde{v}(x,y),\ \ \forall (x,y)\in B_i^x,
\]
  where ${v_h}\in U_{h}(B_{i}^x)$
and $\tilde{v}(x,y)=v(x_{i-1},y)\phi_0(s)+v(x_{i},y)\phi_1(s)$. By
\eqref{correction_1},
\begin{eqnarray*}
    (R_i^x(E^xu),v)
   =-\langle\partial_y^{-1}\partial_xE^xu, \partial^2_{x,y}\tilde{v}\rangle_{B_i^x}- a_{h}(\cal L_{B_i^x}(E^xu), \Pi
   \tilde{v})=-J_1-J_2.
\end{eqnarray*}

  We next estimate $J_1$ and $J_2$ separately.
  Let
  $\Phi=\partial_y^{-1}\partial_xE^xu$. By
  \eqref{eq9}-\eqref{deq} and the fact that
  $\partial^{k+1}_y\Phi=\partial_xE^x(\partial^k_yu)$, we have
  for all $(x,y)\in \tau_{i,j}, (i,j)\in\bZ_m\times\bZ_n$
\begin{eqnarray*}
  |(\Phi-Q_{k}^y\Phi)(x,y)| &\lesssim&
  h^{k}\int_{y_{j-1}}^{y_j}|\partial_xE^x(\partial^k_yu)(x,y)|dy\\
  &\lesssim
  &h^{\alpha-1}\int_{y_{j-1}}^{y_j}\int_{x_{i-1}}^{x_i}|\partial^{\alpha+1-k}_xE^x(\partial^k_yu)(x,y)|dxdy
  \lesssim h^{\alpha}|u|_{\alpha+1,\tau_{i,j}}.
\end{eqnarray*}
  Then by  the Cauchy-Schwartz inequality, we derive
\begin{eqnarray*}
   |\langle\Phi-Q_{k}^y\Phi,
   \partial^2_{x,y}\tilde{v}\rangle_{B_i^x}|&\lesssim& \langle\Phi-Q_{k}^y\Phi,
  \Phi-Q_{k}^y\Phi\rangle_{B_i^x}^{\frac 12}\langle\partial^2_{x,y}\tilde{v} ,
   \partial^2_{x,y}\tilde{v}\rangle_{B_i^x}^{\frac 12}\\
     &\lesssim & h^{\alpha}|u|_{\alpha+1, B_i^x}\|\partial_x\tilde{v}\|_{0,B_i^x}.
\end{eqnarray*}
  Here in the last step, we have used the inverse inequality
\[
 \|\partial^2_{x,y}\tilde{v}\|_{0,B_i^x}\lesssim h^{-1}
  \|\partial_x\tilde{v}\|_{0,B_i^x}.
\]
  Note that $(Q_{k}^y\Phi)
   \partial^2_{x,y}\tilde{v}(x,\cdot)\in \mathbb{P}_{2k-1}$, by Gauss
   quadrature and integrating by part, we obtain
\begin{eqnarray*}
   \langle Q_{k}^y\Phi,
   \partial^2_{x,y}\tilde{v}\rangle_{B_i^x}
  &=&-\sum_{\tau\in
   B_i^x}\sum_{l=1}^kA_{\tau,l}^x\int_{c}^d(\partial_yQ_k^y\Phi)\partial_{x}\tilde{v}(g_{\tau,l}^x,y)dy\\
  &=&-\sum_{\tau\in B_i^x}\sum_{l=1}^kA_{\tau,l}^x\int_{c}^d\partial_xE^x(\partial_yQ_k^y\partial_y^{-1}u)\partial_{x}\tilde{v}(g_{\tau,l}^x,y)dy.
\end{eqnarray*}
  Let $\Upsilon=\partial_yQ_k^y\partial_y^{-1}u$. Since $\tilde v$ is linear with respect to $x$,   we have
\begin{eqnarray*}
\sum_{\tau\in
   B_i^x}\sum_{l=1}^kA_{\tau,l}^x\int_{c}^d(\partial_xQ^x_{\alpha}E^x\Upsilon)\partial_{x}\tilde{v}(g_{\tau,l}^x,y)dy
   =\int_{B_i^x}(\partial_xQ^x_{\alpha}E^x\Upsilon)\partial_{x}\tilde{v}dxdy=0.
\end{eqnarray*}
 Consequently,
\begin{eqnarray*}
  | \langle Q_{k}^y\Phi,
   \partial^2_{x,y}\tilde{v}\rangle_{B_i^x}|&=&\sum_{\tau\in
   B_i^x}\sum_{l=1}^kA_{\tau,l}^x\int_{c}^d\left|\left((\partial_xE^x\Upsilon-\partial_xQ^x_{\alpha}E^x\Upsilon)\partial_{x}\tilde{v}\right)(g_{\tau,l}^x,y)\right|dy\\
   &\lesssim & h^{\alpha}\|\partial^{\alpha+1}_xE^x\Upsilon\|_{0, B_i^x}\|\partial_x\tilde{v}\|_{0,B_i^x}\lesssim h^{\alpha}|u|_{\alpha+1, B_i^x}\|\partial_x\tilde{v}\|_{0,B_i^x}.
\end{eqnarray*}
  Note that
\[
  \partial_x\tilde{v}=\frac{v(x_i,y)-v(x_{i-1},y)}{h}=h^{-1}\int_{x_{i-1}}^{x_i}\partial_x
  v(x,y) dx,
\]
  we have
\[
    \|\partial_x\tilde{v}\|_{0,B_i^x}\lesssim \|v\|_{1,B_i^x}.
\]
   Then
\begin{eqnarray*}
   |J_1| &=& \left|\langle\Phi-Q_{k}^y\Phi,
   \partial^2_{x,y}\tilde{v}\rangle_{B_i^x}+\langle Q_{k}^y\Phi,
   \partial^2_{x,y}\tilde{v}\rangle_{B_i^x}\right|\\
   &\lesssim&
h^{\alpha}|u|_{\alpha+1,B_i^x}\|v\|_{1,B_i^x}.
\end{eqnarray*}

As for $J_2$, recall the  bilinear form $a_h(
 \cdot,\Pi\cdot)$,  and we have
\[
   J_2=-\langle\partial_y^{-1}\partial_x\cal
   L_{B_i^x}(E^xu), \partial^2_{x,y}\tilde{v}\rangle_{B_i^x}-\langle\partial_x^{-1}\partial_y\cal
   L_{B_i^x}(E^xu), \partial^2_{x,y}\tilde{v}\rangle_{B_i^x}
\]
  Note that
\[
   \int_{x_{i-1}}^{x_i}(\partial^2_{x,y}\tilde{v})\partial^{-1}_y\partial_x\cal
   L_{B^x_i}(E^xu) dx=0,
\]
 then
\[
   \langle\partial_y^{-1}\partial_x\cal
   L_{B_i^x}(E^xu), \partial^2_{x,y}\tilde{v}\rangle_{B_i^x}=0.
\]
  Therefore,
\begin{eqnarray*}
    J_2=
   -\langle\partial_x^{-1}\partial_y\cal
   L_{B_i^x}(E^xu),\partial^2_{x,y}\tilde{v}\rangle_{B_i^x}=-\int_{B_i^x}\frac{\partial^2
   \tilde{v}}{\partial x\partial y}\partial_x^{-1}\partial_y\cal L_{B_i^x}(E^xu)
   dxdy.
\end{eqnarray*}
  Since
\[
  \Big(\partial_x^{-1}\partial_y\cal
L_{B_i^x}(E^xu)\Big)(x_i)=\Big(\partial_x^{-1}\partial_y\cal
L_{B_i^x}(E^xu)\Big)(x_{i-1})=0,\ \ \tilde{v}(x,c)=\tilde{v}(x,d)=0,
\]
 integrating by part, we obtain
\begin{eqnarray*}
   J_2&=&-\int_{B_i^x} \tilde{v}\Big(\partial^2_y\cal L_{B_i^x}(E^xu)\Big) dxdy\\
   &=&-\sum_{p,q=2}^kw_{p,q}\int_{B_i^x}\Big(v(x_{i-1},y)\phi_0(s)+v(x_i,y)\phi_1(s)\Big)
   \partial_y^2\Psi_{p,q} dxdy.
\end{eqnarray*}
  Note that $\Psi_{p,q}(\cdot,y)\perp\mathbb{P}_1, p>3$, then only $p=2,3$ in the above equation remain.
  For any $q=2,\ldots,k$, a direct calculation yields
\begin{eqnarray*}
   \left| \int_{B_i^x}\Big(v(x_{i-1},y)\phi_0(s)+v(x_i,y)\phi_1(s)\Big)
   \partial_y^2\Psi_{2,q} dxdy \right|&\lesssim&
   h\int_{c}^d\left|v(x_{i-1},y)+v(x_{i},y)\right| dy\\
   &\lesssim& \int_{B_i^x}\left|v(x,y)\right| dy.
\end{eqnarray*}
  Here in the last step, we have used the inverse inequality
\[
   |v(\xi_i,y)|\lesssim h^{-1}\int_{x_{i-1}}^{x_i}|v(x,y)| dx,\ \ \forall
   \xi_i\in[x_{i-1},x_i], v\in U_h.
\]
  By the same argument, we derive
\begin{eqnarray*}
   \left| \int_{B_i^x}\Big(v(x_{i-1},y)\phi_0(s)+v(x_i,y)\phi_1(s)\Big)
   \partial_y^2\Psi_{3,q} dxdy\right|&\lesssim&
   h\int_{c}^d\left|v(x_{i-1},y)-v(x_{i},y)\right| dy\\
   &\lesssim& h\int_{B_i^x}\left|\partial_xv(x,y)\right| dy.
\end{eqnarray*}
    Substituting the above two inequalities into the formula of $J_2$, we
   have
\begin{eqnarray*}
   |J_2| &\lesssim & (\|X_2\|_{\infty}+h\|X_3\|_{\infty})\|v\|_{1,1,B_i^x}\\
   &\lesssim & h^{\alpha+\frac 12}|{\rm{ln}}h|^{\frac 12}\|u\|_{\alpha+1,B_i^x}\|v\|_{1,B_i^x}.
\end{eqnarray*}
  Then the desired result follows by combining $J_1$ with $J_2$.
\end{proof}

  Similarly, by denoting the residual for all $j\in\bZ_n$
\begin{eqnarray*}
   (R_j^y(w),v)=-\langle\partial^{-1}_y\partial_x w, \partial^2_{x,y}v\rangle_{B_j^y}-a_{h}(\cal L_{B_j^y}(w), \Pi v),\ \
 w\in H^1_0(\Omega), v\in U_h,
\end{eqnarray*}
 we have
\begin{equation}\label{ly-operator}
   |(R^y_j(E^yu),v)|\lesssim
   h^{\alpha}\|u\|_{\alpha+1,B_j^y}\|v\|_{1,B_j^y},
\end{equation}
 and
\begin{equation}\label{ly-operator1}
   |(R^y_j(E^yE^xu),v)|\lesssim
   h^{\alpha}\|u\|_{\alpha+1,B_j^y}\|v\|_{1,B_j^y}.
\end{equation}

\medskip

  With all the  above preparations, we are ready to prove  Proposition \ref{theorem_ver}.


{\it{Proof of Proposition \ref{theorem_ver}}.} As a direct
consequence of \eqref{eq66}, we have
\begin{eqnarray*}
   \left\|\cal L^x(E^xu)\right\|_{\infty}\lesssim
  h^{k+2}|{\rm{ln}}h|^{\frac 12}\|u\|_{\alpha+1}.
\end{eqnarray*}
 Similar results hold true for $\cal
  L^y(E^yu), \tilde{L}^x(E^xu), \tilde{L}^y(E^yu)$ and $\cal
  L^y(E^yE^xu),\tilde{L}^y(E^yE^xu)$ by the same arguments.
  Then \eqref{ineq_1} follows.

  Now we turn to prove \eqref{correction_2}. Let $R=u-u_I$.
   By the orthogonal property, we have for all $v\in
   U_{h}$
\begin{eqnarray*}
   a_{h}(u_h-u_I,\Pi v)&=&a_{h}(u-u_I,\Pi v)\\
  & =&-\langle\partial_y^{-1}\partial_xR,  \partial_{x,y}^2v\rangle-\langle\partial_x^{-1}\partial_yR, \partial_{x,y}^2 v\rangle=I_1+I_2.
\end{eqnarray*}
  From the decomposition \eqref{decomposion}, we have
\[
  I_1=-\langle\partial_y^{-1}\partial_x E^xu, \partial^2_{x,y} v)
  -\langle\partial_y^{-1}\partial_xE^yu, \partial^2_{x,y} v\rangle+\langle\partial_y^{-1}\partial_x(E^yE^xu), \partial^2_{x,y}
  v\rangle.
\]
   Let $w_h=w_1+w_2$ with
\[
   w_1=\cal L^{x}(E^xu)+\cal L^{y}(E^yu)-\cal L^{y}(E^yE^xu),
\]
   and
\[
    w_2=\tilde{\cal L}^{x}(E^xu)+\tilde{\cal L}^{y}(E^yu)-\tilde{\cal
    L}^{y}(E^yE^xu).
\]
  By \eqref{eq7}-\eqref{ly-operator1}, we derive
\begin{eqnarray*}
    |I_1-a_{h}(w_1,\Pi v)|&=&\sum_{B_i^x}\left|(R^x_i(E^xu),v)\right|+\sum_{B_j^y}\left|(R^y_j(E^yu),v)\right|+\left|(R^y_j(E^yE^xu),v)\right|\\
       &\lesssim &h^{\alpha}\|u\|_{\alpha+1}\|v\|_{1}.
\end{eqnarray*}
   By the same arguments, we have
\[
   |I_2-a_{h}(w_2,\Pi v)|\lesssim h^{\alpha}\|u\|_{\alpha+1}\|v\|_{1}.
\]
  Note that
\[
   a_h(u-u_I-w_h,\Pi v)=I_1-a_{h}(w_1,\Pi v)+I_2-a_{h}(w_2,\Pi v),
\]
then \eqref{correction_2} follows.   $\Box$

\section{Superconvergence}
  In this section, we shall study  superconvergence properties of
  $u_h$ at three kinds of special points : nodes, Gauss and Lobatto points.

 Our first goal is to prove the {\it $2k$-conjecture}.

%
%
\begin{theorem}\label{theorem_sup}
   Let $u\in H^{2k+1}(\Omega)$ be the solution of \eqref{Poisson}, and $u_{h}$ the solution of \eqref{FVM}.
   Then,
   \begin{equation}\label{sup-node}
   |(u-u_{h})(P)|\lesssim h^{2k}|{\rm{ln}}h|^{\frac 12}\|u\|_{2k+1}, \forall P\in \N_h.
\end{equation}
\end{theorem}
\begin{proof}
 By \cite{zhang;zou2012}, there hold
\begin{equation}\label{coer}
  a_{h}(w,\Pi v)\lesssim \|w\|_1 \|v\|_1,\
  a_{h}(v,\Pi v)\gtrsim \|v\|_1^2,\ \forall w,v\in U_{h}.
\end{equation}
 For any $v\in U_{h}$ and $Q\in\Omega$, by  the Lax-Milgram Lemma, there exists $g_{h}\in U_{h}$ such that
\begin{equation}\label{equation2}
   a_{h}(v, \Pi g_{h})=v(Q).
\end{equation}
Choosing $v=g_{h}$, we have, from \eqref{coer} and \eqref{equation2}
\[
   \|g_{h}\|_1^2\le|a_{h}(g_{h}, \Pi g_{h})|=|g_{h}(Q)|\le \|g_{h}\|_{\infty}.
\]
  Since (cf.,\cite{Zhu.QD;LinQ1989}, p.84, Theorem 2.8)
\[
   \|v\|_{\infty}\lesssim |{\rm{ln}}h|^{\frac 12}\|v\|_{1},\ \ \forall v\in
   U_{h},
\]
we have
\begin{equation}\label{greenfunction_v}
   \|g_{h}\|_1\lesssim |{\rm{ln}}h|^{\frac 12}.
\end{equation}
  Letting $v=u_h-u_I-w_{h}\in
  U_{h}$ in \eqref{equation2} and using \eqref{correction_2} and
  \eqref{greenfunction_v}, we obtain
\begin{equation}\label{esti1}
   |(u_h-u_I-w_{h})(Q)|
   =|a_{h}(u-u_I-w_{h}, \Pi g_{h
   })|\lesssim h^{2k}|{\rm{ln}}h|^{\frac 12}\|u\|_{2k+1}.
\end{equation}
   Noticing $w_h=0$ and $u_I=u$ at all nodes $P\in \N_h$,
 the desired result \eqref{sup-node} follows.
\end{proof}

We next discuss superconvergence of $u_h$ at Gauss and Lobatto points.

\begin{theorem}\label{core_0}
    Let $u\in H^{k+3}(\Omega)$ be the solution of \eqref{Poisson}, and $u_{h}$ the solution of \eqref{FVM}.
   Then,
\begin{equation}\label{sup-lobatto}
   |(u-u_{h})(P)|\lesssim h^{k+2}|{\rm{ln}}h|^{\frac
   {1}{2}}\|u\|_{k+3},\ \ \forall
   P\in\N^l,
\end{equation}
   and
\begin{equation}\label{sup-gauss}
   |\nabla(u-u_{h})(Q)|\lesssim
   h^{k+1}|{\rm{ln}}h|^{\frac
   {1}{2}}\|u\|_{k+3},\ \ \forall Q\in\cal\N^g.
\end{equation}
\end{theorem}
\begin{proof}
 By \eqref{ineq_1}-\eqref{correction_2} and \eqref{equation2}, we have
\begin{eqnarray*}
  \|u_I-u_h\|_{\infty}\lesssim h^{k+2}|{\rm{ln}}h|^{\frac
12}\|u\|_{k+3}.
\end{eqnarray*}
  By the inverse inequality,
\[
   |u_I-u_h|_{1,\infty}\lesssim h^{-1}\|u_I-u_h\|_{\infty}\lesssim h^{k+1}|{\rm{ln}}h|^{\frac
12}\|u\|_{k+3}.
\]
 On the other hand, by the definition of $u_I$, we have (see, e.g.,\cite{Chen.C.M2001,Zhu.QD;LinQ1989})
\[
   |(u-u_{I})(P)|\lesssim
   h^{k+2}|u|_{k+2,\infty},\ \  \forall
   P\in\N^l,
\]
  and
\[
   |\nabla(u-u_{I})(Q)|\lesssim
   h^{k+1}|u|_{k+2,\infty},\ \ \forall Q\in \N^g.
\]
  The desired statements  \eqref{sup-lobatto}-\eqref{sup-gauss} then follows.
\end{proof}

\begin{remark}
   As a direct consequence of the above theorem, we have
\[
    |u_h-u_I|_{1}\lesssim |u_h-u_I|_{1,\infty}\lesssim h^{k+1}|{\rm{ln}}h|^{\frac
  12}\|u\|_{k+3},
\]
  and
\[
   \|u_I-u_h\|_{0}\lesssim \|u_I-u_h\|_{\infty} \lesssim h^{k+2}|{\rm{ln}}h|^{\frac
12}\|u\|_{k+3}.
\]
   It was pointed out in \cite{zhang;zou2012} that the FV approximation $u_h$ is
  super-close to the Lobatto interpolation function $\tilde{u}_I$.
 The above inequalities clearly indicate the same for the interpolation function $u_I$, i.e.,
  $u_h$ is also super-close to $u_I$ up to a logarithmic factor.
\end{remark}

  \section{Numerical results}

  In this section, we present numerical examples to support our theoretical findings in the previous section.

  We consider \eqref{Poisson} with $\Omega=[0,1]\times[0,1]$ and the
  right-hand side
\begin{eqnarray*}
   f(x,y) &=& [ (5\pi^2-4y^2-3)\sin(\pi x)\sin(2\pi y) -8\pi y\sin(\pi x)\cos(2\pi y) \\
   &&  -2\pi \cos(\pi x)\sin(2\pi y) ] e^{x-0.5+y^2}.
\end{eqnarray*}
   The exact solution is then
\[
    u(x,y) = \sin(\pi x) \sin(2\pi y) e^{x-0.5+y^2}, \; (x,y)\in\Omega.
\]

   We construct ${\cal T}_h$ with $h = 2^{-s}$, $s=1,2,\ldots,8$, by dividing $\Omega$ into
  $h^{-1}\times h^{-1}$ squares, and solve \eqref{Poisson} by the FV scheme \eqref{FVM} with $k=3,4$.
  For each $h$ and $k$, we measure maximum errors at nodes, Lobatto points, and Gauss points (for gradient only), respectively.
  They are defined by
\[
 e_{N}=\max_{P\in\N_h}|(u-u_h)(P)|, \;\;  e_{L}=\max_{P\in\N^ l}|(u-u_h)(P)|, \;\;
e_{G}=\max_{Q\in\N^ g}|\nabla(u-u_h)(Q)|.
\]

 Numerical data are demonstrated in Table \ref{case=0}, and corresponding error curves are depicted in Figure \ref{1_2} with log-log scale.
  We observe a convergence slope $k+1$ for $e_G$, $k+2$ for $e_L$, and $2k$ for $e_N$, respectively. These results
   confirm our theoretical findings in Theorem \ref{core_0} and Theorem \ref{theorem_sup}: The derivative error
  is superconvergent at all Gauss points and the function value error is superconvergent at all Lobatto points. Moreover,
 the approximation error at nodes converges with a rate $h^{2k}$, the {\it $2k$-conjecture} for our finite volume approximation is verified.

 \begin{table}
\caption{} \centering
\begin{threeparttable}
        \begin{tabular}{|c|c|c|c||c|c|c|}
        \hline
        &\multicolumn{3}{|c||}{$k=3$}&\multicolumn{3}{|c|}{$k=4$}\\
        \cline{1-7}N & $e_{G}$ &  $e_{L}$ & $e_{N}$ & $e_{G}$ &  $e_{L}$ & $e_{N}$ \\
        \hline
2 &  2.699e-1 &  8.851e-3 &  1.327e-3 &  4.326e-2 &  2.044e-4 & 1.190e-5\\
4 &  2.897e-2 &  2.902e-4 &  2.761e-5 & 1.536e-3 &  5.354e-6 &  8.178e-8\\
8 &  2.224e-3 &  8.863e-6 &  5.743e-7 & 4.979e-5 &  1.092e-7 &  3.750e-10\\
16 & 1.660e-4 &  2.701e-7 &  1.056e-8 & 1.586e-6 &  2.397e-9  & 1.510e-12\\
32 & 1.117e-5 &  8.288e-9 &  1.919e-10 & 4.986e-8&  4.340e-11 &  ---\\
64 &  7.222e-7 & 2.567e-10 &  3.309e-12& 1.564e-9 & 7.257e-13  & ---\\
\hline
 \end{tabular}
 \end{threeparttable}
 \label{case=0}
\end{table}

\begin{figure}[htbp]
\scalebox{0.5}{\includegraphics{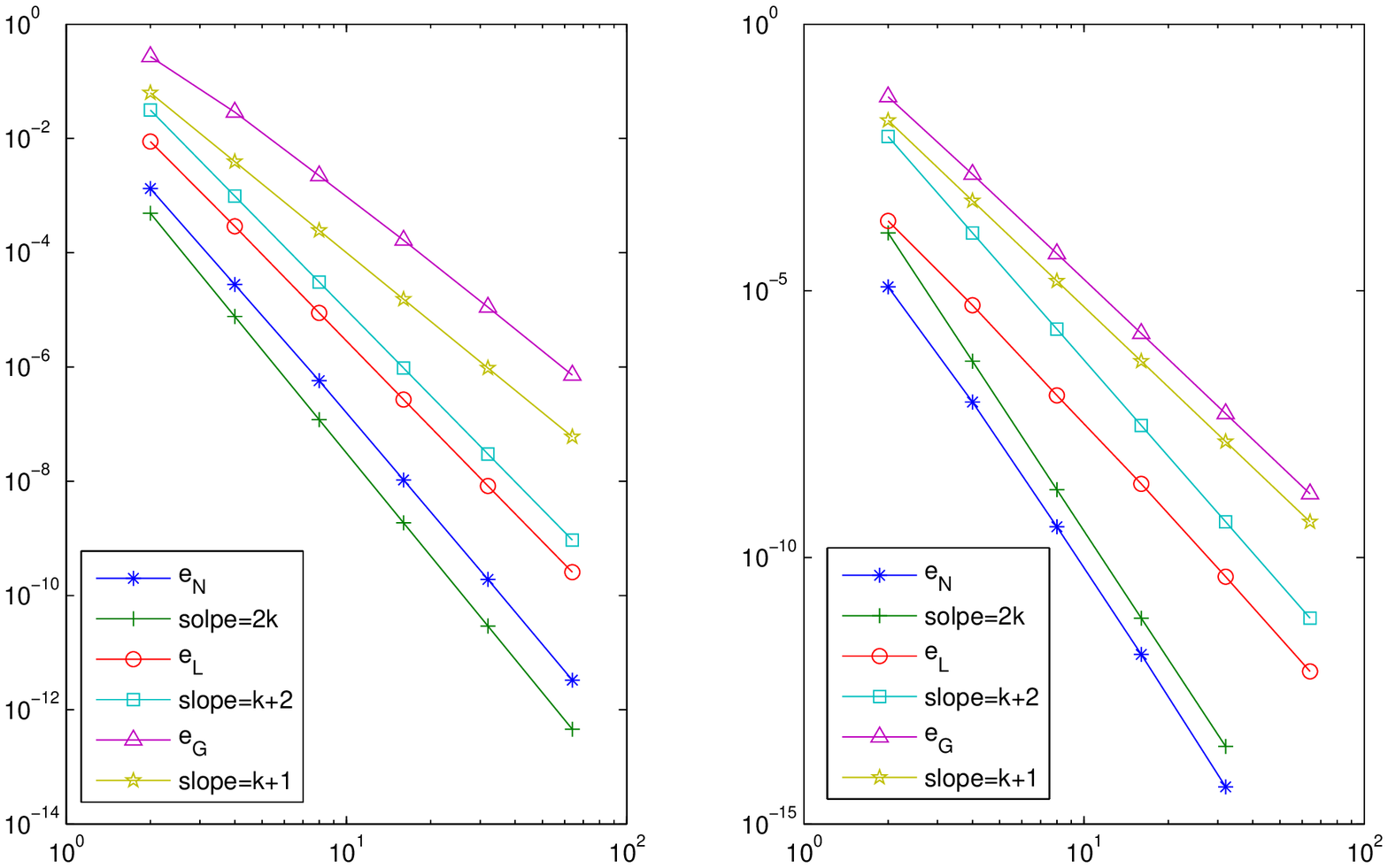}}
 \caption{left: $k=3$, right: $k=4$.}\label{1_2}
\end{figure}


\bigskip
\bigskip
\end{document}